\documentclass[portuges,12pt,letter]{article}
\usepackage[centertags]{amsmath}
\usepackage{amsfonts}
\usepackage{newlfont}
\usepackage{amscd}
\usepackage{graphics}
\usepackage{epsfig}
\usepackage{indentfirst}
\usepackage{amsxtra}
\usepackage[latin1]{inputenc}
\usepackage{amssymb, amsmath}
\usepackage{amsthm}
\usepackage[mathscr]{eucal}

\newtheorem{thm}{Theorem}[section]

\newtheorem{dfn}[thm]{Definition}

\newtheorem{remark}[thm]{Remark}

\author{Fabio Silva Botelho \\ Department of Mathematics \\  Federal University of Santa Catarina, UFSC \\
Florian\'{o}polis, SC - Brazil}

\title{\bf  A general variational formulation for relativistic mechanics based on fundamentals of differential geometry}
\date{}
\begin{document}
\maketitle

\abstract{ The first part of this article develops a variational formulation for relativistic mechanics. The results are established through standard tools of variational analysis and
differential geometry. The novelty here is that the main motion manifold has  a $n+1$ dimensional range.
It is worth emphasizing  in a first approximation we have
neglected the self-interaction energy part.
In its second part, this article develops some formalism concerning the causal structure in a general space-time manifold.
Finally, the last article section presents a result concerning the existence of a generalized solution for the world sheet  manifold variational formulation.}

\section{Introduction} Let $\Omega \subset \mathbb{R}^3$ be an open, bounded, connected set with a smooth boundary (at least $C^1$ class) denoted by
$\partial \Omega$ and let $[0,T]$ be a time interval. Consider a relativistic motion given by a position field $$(\mathbf{r} \circ \hat{\mathbf{u}}):\Omega \times [0,T]
\rightarrow \mathbb{R}^{n+1}.$$

Here, for an open, bounded and connected  set $D$ with a smooth boundary, we consider a world sheet smooth ($C^3$ class) manifold $\mathbf{r}:D \subset \mathbb{R}^{m+1} \rightarrow \mathbb{R}^{n+1}$,  where point-wise
$$\mathbf{r}(\hat{\mathbf{u}})=(ct,X_1(\mathbf{u}), \ldots, X_{n}(\mathbf{u}))$$
and where
$$\mathbf{r}(\hat{\mathbf{u}}(\mathbf{x},t))=(u_0(\mathbf{x},t),X_1(\mathbf{u}(\mathbf{x},t)), \ldots, X_{n}(\mathbf{u}(\mathbf{x},t))),$$

$$\hat{\mathbf{u}}(\mathbf{x},t)=(u_0(\mathbf{x},t),u_1(\mathbf{x},t), \ldots, u_m(\mathbf{x},t)),$$
 $1 \leq m <n$
and
$$u_0(\mathbf{x},t)=ct.$$

Consider also a density scalar field given by
$$m |\phi(\mathbf{u})|^2:\Omega \times [0,T] \rightarrow \mathbb{R}^+,$$
where $m$ is the total system mass and $$\phi: D \subset \mathbb{R}^{m+1} \rightarrow \mathbb{C}$$  is a wave function.

At this point we highlight that
$$\frac{d \mathbf{r}(\mathbf{u}(\mathbf{x},t))}{d t} \cdot \frac{d \mathbf{r}(\mathbf{u}(\mathbf{x},t))}{d t}=-c^2+\sum_{j=1}^n \left(\frac{ d X_j(\mathbf{u}(\mathbf{x},t))}{dt}\right)^2=-c^2+v^2,$$
where $c$ denotes the speed of light at vacuum and
$$v=\sqrt{\sum_{j=1}^n \left(\frac{d X_j(\mathbf{u}(\mathbf{x},t))}{dt}\right)^2}.$$

We also emphasize that generically, for $a=(\hat{x}_0,x_1,x_2,x_3) \in \mathbb{R}^4$ and $b=(\hat{y}_0,y_1,y_2,y_3)\in \mathbb{R}^4$ we have
$$a\cdot b=-\hat{x}_0\hat{y}_0+\sum_{j=1}^3x_iy_i.$$

Moreover $x_0=t$, $\mathbf{x}=(x_1,x_2,x_3)$ and $$d\mathbf{x}=dx_1dx_2dx_3.$$

Finally, we generically refer to $$(\mathbf{r} \circ \hat{\mathbf{u}}):\Omega \times [0,T] \rightarrow \mathbb{R}^{n+1}$$ as a space-time manifold. Furthermore, with such a notation in mind we denote
\begin{eqnarray}ds^2&=&d\mathbf{r}(\mathbf{u}(\mathbf{x},t)) \cdot d\mathbf{r}(\mathbf{u}(\mathbf{x},t)) \nonumber \\ &=&-c^2dt^2+([dX_1(\mathbf{u}(\mathbf{x},t))]^2+[dX_2(\mathbf{u}(\mathbf{x},t))]^2+\ldots+[dX_n(\mathbf{u}(\mathbf{x},t))]^2).\end{eqnarray}
\begin{remark} About the references, the mathematical background necessary may be found in \cite{19,12}. For the part on relativistic physics, we follow at some extent, the references \cite{550,50}.
\end{remark}
\section{ The system energy}

Consider first the mass differential, given by,
$$dm=\frac{m|\phi(\mathbf{u}(\mathbf{x},t))|^2}{\sqrt{1-\frac{v^2}{c^2}}}\sqrt{-g} \sqrt{U}\;d\mathbf{x},$$
so that the kinetics energy differential is defined by
\begin{eqnarray}dE_c&=& \frac{d \mathbf{r}}{d t} \cdot \frac{d \mathbf{r}}{d t}\; dm
\nonumber \\ &=& - \frac{c^2-v^2}{\sqrt{1-\frac{v^2}{c^2}}} m|\phi|^2\sqrt{g} \sqrt{U}\;d \mathbf{x}
\nonumber \\ &=& -mc\sqrt{c^2-v^2} |\phi|^2\sqrt{g} \sqrt{U}\;d \mathbf{x}
\nonumber \\ &=& -mc \sqrt{ -\frac{d \mathbf{r}}{d t} \cdot \frac{d \mathbf{r}}{d t}} |\phi|^2\sqrt{-g} \sqrt{U}\;d\mathbf{x}
\nonumber \\ &=&-mc \sqrt{ -\frac{\partial \mathbf{r}}{\partial u_j} \frac{\partial u_j}{\partial t} \cdot \frac{\partial \mathbf{r}}{\partial u_k} \frac{\partial u_k}{\partial t}} |\phi|^2\sqrt{-g} \sqrt{U}\;d\mathbf{x}
\nonumber \\ &=& -mc |\phi|^2 \sqrt{-g_{jk} \frac{\partial u_j}{\partial t} \frac{\partial u_k}{\partial t}}\sqrt{-g} \sqrt{U}\;d\mathbf{x}.
\end{eqnarray}
Where $$\mathbf{g}_j=\frac{\partial \mathbf{r}(\mathbf{u})}{\partial u_j},\; \forall j \in \{0,\ldots,m\},$$
$$g_{jk}=\mathbf{g}_j \cdot \mathbf{g}_k,\; \forall j,k \in \{0,\ldots,m\},$$ $$\{g^{jk}\}=\{g_{jk}\}^{-1},$$ $$g=\det\{g_{ij}\}$$ and $$U_{ij}=\frac{\partial \mathbf{u}(\mathbf{x},t)}{\partial x_i} \cdot
\frac{\partial \mathbf{u}(\mathbf{x},t)}{\partial x_j},\; \forall i,j \in \{0,1,2,3\}.$$ Moreover, we define $$U=|\det\{U_{ij}\}|.$$

At this point, we assume there exists a  smooth normal field $\mathbf{n}$ such that $$\text{Span}\left\{ \left\{\frac{\partial \mathbf{r}(\mathbf{u})}{\partial u_j},\;\forall  j \in \{0,\ldots,m\}\right\},\mathbf{n}(\mathbf{u})\right\}\subset \mathbb{R}^{n+1},\; \forall \mathbf{{u}} \in D$$ and $$\frac{\partial^2 \mathbf{r}(\mathbf{u})}{\partial u_j\partial u_k}=\Gamma_{jk}^l(\mathbf{u}) \frac{\partial \mathbf{r}(\mathbf{u})}{\partial u_l}+b_{jk}(\mathbf{u}) \mathbf{n}(\mathbf{u}),\; \forall \mathbf{u} \in D,$$
where $\{\Gamma_{jk}^l\}$ are the Christoffel symbols and the concerning normal field $\mathbf{n}(\mathbf{u})$ is also such that
$$\mathbf{n}(\mathbf{u}) \cdot \mathbf{n}(\mathbf{u})=1,\; \forall \mathbf{u} \in D,$$ $$\frac{\partial \mathbf{r}(\mathbf{u})}{\partial u_l} \cdot \mathbf{n}(\mathbf{u})=0, \text{ in } D, \forall l \in \{0,\ldots,m\}$$ and $$b_{jk}(\mathbf{u})=\frac{\partial^2 \mathbf{r}(\mathbf{u})}{\partial u_j \partial u_k}\cdot \mathbf{n}(\mathbf{u}),\; \forall \mathbf{u} \in D, \;\forall j,k \in \{0,\ldots,m\}.$$

Suppose also the concerning world sheet position field is such that there exist smooth normal fields
$$\hat{\mathbf{n}_1},\ldots, \hat{\mathbf{n}}_s$$ where $m+1+s \geq n+1$ such that
$$\text{Span}\left\{ \left\{\frac{\partial \mathbf{r}(\mathbf{u})}{\partial u_j},\;\forall  j \in \{0,\ldots,m\}\right\},\hat{\mathbf{n}}_1(\mathbf{u}),\ldots,\hat{\mathbf{n}}_s(\mathbf{u})\right\}= \mathbb{R}^{n+1},\; \forall \mathbf{{u}} \in D$$
so that $$\mathbf{n}(\mathbf{u})=f_q(\mathbf{u})\hat{\mathbf{n}}_q(\mathbf{u}),\; \forall \mathbf{u} \in D$$ for an appropriate field $\{f_q\}_{q=1}^s.$

Moreover, we assume
$$\hat{\mathbf{n}}_j(\mathbf{u}) \cdot \hat{\mathbf{n}}_k(\mathbf{u})=\delta_{jk},\; \forall \mathbf{u} \in D,\; j,k \in \{1, \ldots,s\}$$
and  $$\frac{\partial \mathbf{r}(\mathbf{u})}{\partial u_j} \cdot \hat{\mathbf{n}}_k(\mathbf{u})=0,$$
$\forall \mathbf{u} \in D, \; \forall j \in \{0,\ldots,m\},\; k \in \{1,\ldots,s\}.$

Here we recall that
$$\mathbf{n}(\mathbf{u}) \cdot \frac{\partial \mathbf{r}(\mathbf{u})}{\partial u_k}=0, \text{ in } D.$$

Hence,
$$ \frac{\partial \mathbf{n}(\mathbf{u})}{\partial u_j}\cdot  \frac{\partial \mathbf{r}(\mathbf{u})}{\partial u_k}+
\mathbf{n}(\mathbf{u}) \cdot \frac{\partial^2 \mathbf{r}(\mathbf{u})}{\partial u_j \partial u_k}=0,$$
that is,
\begin{equation}\label{tc36} \frac{\partial \mathbf{n}(\mathbf{u})}{\partial u_j}\cdot  \frac{\partial \mathbf{r}(\mathbf{u})}{\partial u_k}=-b_{jk}.\end{equation}
We may also denote
$$\frac{\partial \mathbf{n}(\mathbf{u})}{\partial u_j}=c_j^s \frac{\partial \mathbf{r}(\mathbf{u})}{\partial u_s}+e_j^q \hat{\mathbf{n}}_q,$$
for an appropriate $\{c_j^s\}$ and where $$e_j^q=\frac{\partial \mathbf{n}(\mathbf{u})}{\partial u_j}\cdot \hat{\mathbf{n}}_q.$$

From this and (\ref{tc36}), we obtain

$$c_j^s \frac{\partial \mathbf{r}(\mathbf{u})}{\partial u_s}\cdot \mathbf{g}_k=c_j^s g_{sk}=-b_{jk},$$
so that
$$c_j^s g_{sk} g^{kl}=-b_{jk}g^{kl}=-b_j^l,$$
that is,
$$c_j^l=c_j^s \delta_s^l=- b_j^l,$$
where
$$b_j^l=b_{jk}g^{kl}.$$

Summarizing, we have got
$$\frac{\partial \mathbf{n}(\mathbf{u})}{\partial u_j}=-b_j^l \frac{\partial \mathbf{r}(\mathbf{u})}{\partial u_l}+e_j^q \hat{\mathbf{n}}_q.$$
Observe now that
\begin{eqnarray}\frac{\partial^3 \mathbf{r}(\mathbf{u})}{\partial u_i \partial u_j\partial u_k}&=&
\frac{\partial}{\partial u_i} \left(\Gamma_{jk}^l \frac{\partial \mathbf{r}(\mathbf{u})}{\partial u_l}+b_{jk} \mathbf{n} \right)
\nonumber \\ &=& \left(\frac{\partial \Gamma_{jk}^l}{\partial u_i}+\Gamma_{jk}^p\Gamma_{pi}^l \right)\frac{\partial \mathbf{r}(\mathbf{u})}{\partial u_l}
\nonumber \\ &&+ \Gamma_{jk}^p b_{p i} \mathbf{n}+\frac{\partial b_{jk}}{\partial u_i} \mathbf{n}- b_{jk}b_i^l \frac{\partial \mathbf{r}(\mathbf{u})}{\partial u_l} \nonumber \\ &&+ b_{jk}e^l_i \hat{\mathbf{n}}_l.\end{eqnarray}

Similarly
\begin{eqnarray}\frac{\partial^3 \mathbf{r}(\mathbf{u})}{\partial u_j \partial u_i\partial u_k}&=&
\frac{\partial}{\partial u_j} \left(\Gamma_{ik}^l \frac{\partial \mathbf{r}(\mathbf{u})}{\partial u_l}+b_{ik} \mathbf{n} \right)
\nonumber \\ &=& \left(\frac{\partial \Gamma_{ik}^l}{\partial u_j}+\Gamma_{ik}^p\Gamma_{pj}^l\right) \frac{\partial \mathbf{r}(\mathbf{u})}{\partial u_l}
\nonumber \\ &&+ \Gamma_{jk}^p b_{p i} \mathbf{n}+\frac{\partial b_{ik}}{\partial u_j} \mathbf{n}- b_{ik}b_j^l \frac{\partial \mathbf{r}(\mathbf{u})}{\partial u_l}\nonumber \\ &&+ b_{ik}e^l_j \hat{\mathbf{n}}_l.\end{eqnarray}

Thus, for such a smooth ($C^3$ class) manifold, from
$$\frac{\partial^3 \mathbf{r}(\mathbf{u})}{\partial u_i \partial u_j\partial u_k}=\frac{\partial^3 \mathbf{r}(\mathbf{u})}{\partial u_j \partial u_i\partial u_k},$$
assuming a concerning linear independence and equating the terms in $$\frac{\partial \mathbf{r}(\mathbf{u})}{\partial u_l},$$ we get
\begin{eqnarray}
W_{ijk}^l &=& b_{jk}b_i^l \nonumber \\ &=& \frac{\partial \Gamma_{jk}^l}{\partial u_i}-\frac{\partial \Gamma_{ik}^l}{\partial u_j}
\nonumber \\ &&+\Gamma_{jk}^p\Gamma_{pi}^l-\Gamma_{ik}^p\Gamma_{pj}^l+b_{ik}b_j^l.
\end{eqnarray}

Defining the Riemann curvature tensor by
\begin{eqnarray}R_{ijk}^l=\frac{\partial \Gamma_{jk}^l}{\partial u_i}-\frac{\partial \Gamma_{ik}^l}{\partial u_j}
+\Gamma_{jk}^p\Gamma_{pi}^l-\Gamma_{ik}^p\Gamma_{pj}^l,\end{eqnarray}
we also define the energy part $J_1(\phi,\mathbf{r},\mathbf{u},\mathbf{n})$ as
\begin{eqnarray}
J_1(\phi,\mathbf{r},\mathbf{u},\mathbf{n})&=& \frac{1}{2}\int_0^T\int_\Omega |\phi|^2g^{jk}b_{jl}b_{k}^l \sqrt{-g} \sqrt{U}\;d\mathbf{x}\;dt
\nonumber \\ &=& \frac{1}{2} \int_0^T\int_\Omega |\phi|^2g^{jk} R_{jlk}^l \sqrt{-g} \sqrt{U}\;d\mathbf{x}\;dt \nonumber \\ &&+
\frac{1}{2}\int_0^T\int_\Omega |\phi|^2g^{jk} b_{jk}b_l^l\sqrt{-g} \sqrt{U}\;d\mathbf{x}\;dt.
\end{eqnarray}

The next energy part is defined through the tensor $S_{ijk}^l$ which, considering the Levi-Civita connection $\nabla$ and the standard Lie Bracket $[\cdot,\cdot]$ (see \cite{19,550} for more details),  is such that
$$\nabla_{\left[\phi \frac{\partial \mathbf{r}(\mathbf{u})}{\partial u_i}, \frac{\partial \mathbf{r}(\mathbf{u})}{\partial u_j}\right]} \left( \phi^*
\frac{\partial \mathbf{r}(\mathbf{u})}{\partial u_k} \right)=S_{ijk}^l \frac{\partial \mathbf{r}(\mathbf{u})}{\partial u_l}+\hat{b}_{ijk} \mathbf{n}.$$

Observe that
\begin{eqnarray}&&
\nabla_{\left(\frac{\partial \phi}{\partial u_j} \frac{\partial \mathbf{r}(\mathbf{u})}{\partial u_i}\right) }\left(\phi^* \frac{\partial \mathbf{r}(\mathbf{u})}{\partial u_k}\right)
\nonumber \\ &=& \frac{\partial \phi}{\partial u_j}\frac{\partial \phi^*}{\partial u_i} \frac{\partial \mathbf{r}(\mathbf{u})}{\partial u_k}
\nonumber \\ &&+\frac{\partial \phi}{\partial u_j} \phi^* \frac{\partial^2 \mathbf{r}(\mathbf{u})}{\partial u_i \partial u_k}
\nonumber \\ &=& \frac{\partial \phi}{\partial u_j} \frac{\partial \phi^*}{\partial u_i} \frac{\partial \mathbf{r}(\mathbf{u})}{\partial u_l} \delta_{lk}
\nonumber \\ &&+ \frac{\partial \phi}{\partial u_j} \phi^* \left( \Gamma_{ik}^l \frac{\partial \mathbf{r}(\mathbf{u})}{\partial u_l}+ b_{ik} \mathbf{n}\right)
\end{eqnarray}

Thus,
$$S_{ijk}^l=\frac{\partial \phi}{\partial u_j}\frac{\partial \phi^*}{\partial u_i} \delta_{kl}+\frac{\partial \phi}{\partial u_i}\phi^* \Gamma_{jk}^l.$$

With such results in mind, we define this energy part as
$$J_2(\phi,\mathbf{r},\mathbf{u}, \mathbf{n})=\frac{1}{2}\int_0^T\int_\Omega g^{jk} Re[S_{jlk}^l] \sqrt{-g} \sqrt{U}\;d\mathbf{x}\;dt,$$
where generically $Re[z]$ and $z^*$ denote the real part and complex conjugation, respectively, of $z \in \mathbb{C}.$

\section{The final energy expression}

The expression for the energy, already including the  Lagrange multiplier concerning the mass restriction, is given by
\begin{eqnarray}J(\phi,\mathbf{r},\mathbf{u},\mathbf{n},E)&=&-\int_0^T\int_\Omega dE_c\;dt+J_1(\phi,\mathbf{r},\mathbf{u},\mathbf{n})+J_2(\phi,\mathbf{r},\mathbf{u},\mathbf{n})
\nonumber \\ && -\int_0^T E(t)\left( \int_\Omega |\phi|^2\sqrt{-g} \sqrt{U}\;d\mathbf{x}-1\right)\;dt,\end{eqnarray} so that
\begin{eqnarray}
J(\phi,\mathbf{r},\mathbf{n},\mathbf{u},E)&=& \int_0^T\int_\Omega mc |\phi|^2 \sqrt{-g_{jk} \frac{\partial u_j}{\partial t} \frac{\partial u_k}{\partial t}}
\sqrt{-g} \sqrt{U}\;d\mathbf{x}\;dt \nonumber \\ &&
+\frac{1}{2} \int_0^T\int_\Omega |\phi|^2g^{jk} b_{jl}b_k^l\sqrt{-g} \sqrt{U}\;d\mathbf{x}\;dt \nonumber \\ &&
+ \frac{1}{2}\int_0^T\int_\Omega g^{jk} \frac{\partial \phi}{\partial u_j} \frac{\partial \phi^*}{\partial u_k} \sqrt{-g}\sqrt{U}\;d\mathbf{x}\;dt
\nonumber \\ && +\frac{1}{4} \int_0^T \int_\Omega \left(\frac{\partial \phi}{\partial u_l}\phi^*+\frac{\partial \phi^*}{\partial u_l}\phi \right)\Gamma_{jk}^l g^{jk} \sqrt{-g}\sqrt{U}\;d\mathbf{x}\;dt
\nonumber \\ &&
-\int_0^T E(t)\left( \int_\Omega |\phi|^2\sqrt{-g} \sqrt{U}\;d\mathbf{x}-1\right)\;dt\end{eqnarray}
We shall look for critical points subject to

$$\mathbf{n}(\mathbf{u}(\mathbf{x},t)) \cdot \mathbf{n}(\mathbf{u}(\mathbf{x},t))=1, \text{ in } \Omega \times [0,T]$$
and
$$\frac{\partial \mathbf{r}(\mathbf{u}(\mathbf{x},t))}{\partial u_j} \cdot \mathbf{n}(\mathbf{u}(\mathbf{x},t))=0, \text{ in } \Omega \times [0,T],\; \forall j \in \{0,\ldots,m\}.$$

Already including the concerning Lagrange multipliers, the final functional expression would be
\begin{eqnarray}
J(\phi,\mathbf{r},\mathbf{u},\mathbf{n},E,\lambda)&=& \int_0^T\int_\Omega mc |\phi|^2 \sqrt{-g_{jk} \frac{\partial u_j}{\partial t} \frac{\partial u_k}{\partial t}}
\sqrt{-g} \sqrt{U}\;d\mathbf{x}\;dt \nonumber \\ &&
+\frac{1}{2} \int_0^T\int_\Omega |\phi|^2g^{jk} b_{jl}b_k^l\sqrt{-g} \sqrt{U}\;d\mathbf{x}\;dt \nonumber \\ &&
+ \frac{1}{2}\int_0^T\int_\Omega g^{jk} \frac{\partial \phi}{\partial u_j} \frac{\partial \phi^*}{\partial u_k} \sqrt{-g}\sqrt{U}\;d\mathbf{x}\;dt
\nonumber \\ && +\frac{1}{4} \int_0^T \int_\Omega \left(\frac{\partial \phi}{\partial u_l}\phi^*+\frac{\partial \phi^*}{\partial u_l}\phi\right) \Gamma_{jk}^l g^{jk} \sqrt{-g}\sqrt{U}\;d\mathbf{x}\;dt
\nonumber \\ &&
-\int_0^T E(t)\left( \int_\Omega |\phi|^2\sqrt{-g} \sqrt{U}\;d\mathbf{x}-1\right)\;dt
\nonumber \\ && +\sum_{j=0}^m\int_0^T\int_\Omega \lambda_j(\mathbf{x},t) \frac{\partial \mathbf{r}(\mathbf{u}(\mathbf{x},t))}{\partial u_j} \cdot \mathbf{n}(\mathbf{u}(\mathbf{x},t)) \sqrt{-g}\sqrt{U}\;d\mathbf{x}\;dt \nonumber \\ &&+\int_0^T\int_\Omega \lambda_{m+1}(\mathbf{x},t)
(\mathbf{n}(\mathbf{u}(\mathbf{x},t)) \cdot \mathbf{n}(\mathbf{u}(\mathbf{x},t))-1) \sqrt{-g}\sqrt{U}\;d\mathbf{x}\;dt
\end{eqnarray}
\begin{remark} We must consider such a functional defined on a space of sufficiently smooth functions with appropriate boundary and initial conditions
prescribed.

 Finally, the main difference concerning standard differential geometry in $\mathbb{R}^3$ is that, since $$1 \leq m < n,$$ we have to obtain
through the variation of $J$, the optimal normal field $\mathbf{n}$. Summarizing, at first we do not have an explicit expression for such a field.
\end{remark}

\section{Causal structure}
In this section we develop some formalism concerning the causal structure in a space-time manifold defined by a function $$(\mathbf{r} \circ \hat{\mathbf{u}}):\Omega \times (-\infty,+\infty)
\rightarrow \mathbb{R}^{n+1},$$
where $ t \in (-\infty,+\infty)$ denotes time.

We follow at some extent, the content in the Wald's book \cite{598}, where more details may be found.

\begin{dfn} Let $M$ be a space-time manifold time oriented, in the sense that the light cone related to the tangent
spaces varies smoothly along $M$. A  $C^1$ class curve $\lambda:[a,b] \rightarrow M$ is said to be time-like future directed if for each
$p \in \lambda$ the respective tangent vector is time-like future directed, that is, $$\frac{d\lambda(s)}{ds} \cdot   \frac{d \lambda(s)}{ds}<0,\; \forall s \in [a,b],\text{ (time-like condition)}$$ and $$\frac{d t(s)}{d s}>0, \forall s \in [a,b], \text{(future directed condition)}.$$

Here $$\lambda(s)=\mathbf{r}(\mathbf{\hat{u}}(\mathbf{x}(s),t(s)))$$ for  appropriate smooth functions $$\mathbf{x}(s),t(s).$$

Similarly, we say that such a curve is causal future directed, if the tangent vector is a time-like future directed or is  a null vector, $\forall s \in [a,b].$

Finally, in an analogous fashion we may define a continuous and piece-wise $C^1$ class time-like future directed curve.
\end{dfn}
\begin{remark} At this point we highlight that in the next lines the norm $\|\cdot\|$ refers to the standard Euclidean one in $\mathbb{R}^{n+1}.$
\end{remark}
\begin{dfn} The chronological future of $p \in M$, denoted by $I^+(p)$, is defined as
\begin{eqnarray}I^+(p)&=&\{q \in M\;:\;  \nonumber \\ &&\text{ there exists a continuous and piece-wise  $C^1$ class time-like } 
 \nonumber \\ &&\text{ future directed curve } \lambda:[a,b] \rightarrow M
\nonumber \\ && \text{ such that } \lambda(a)=p \text{ and } \lambda(b)=q\}.\end{eqnarray}
\end{dfn}
Observe that, if $M$ is smooth (as previously indicated, the world sheet manifold in question is at least $C^3$ class) by continuity, if $q \in I^+(p)$ there exists a neighborhood $\mathcal{O}(q)$ such that $$\mathcal{O}(q) \cap M \subset I^+(p).$$

From now and on we always assume any space-time mentioned is always smooth and time-oriented.

Also, for $S \subset M$, we define
$$I^+(S)=\cup_{p \in S} I^+(p),$$ so that
since $I^+(p)$ is open for each $p \in M$, we may infer that $I^+(S)$ is open.
\begin{remark} Similarly, we define the chronological pasts $I^-(p)$ and $I^-(S).$

Moreover the causal future of $p \in M$, denoted by $J^+(p)$ is defined as
\begin{eqnarray}J^+(p)&=&\{q \in M\;:\;  \nonumber \\ &&\text{ there exists a continuous and piece-wise $C^1$ class }
\nonumber \\ && \text{ casual future directed curve } \lambda:[a,b] \rightarrow M
\nonumber \\ && \text{ such that } \lambda(a)=p \text{ and } \lambda(b)=q\}.\end{eqnarray}
Also, we define $$J^+(S)=\cup_{ p \in S} J^+(p),$$ and similarly define the causal pasts $J^-(p)$ and $J^-(S).$
\end{remark}
\begin{dfn} Let $M$ be a space time manifold. We say that $M$ is normal if for each connected set $S \subset M$, there exists $r>0$ such that
 if $p,q \in I^+(S)$ and $0<\|p-q\|<r$, then, interchanging the roles of $p$ and $q$ if necessary, there exists a smooth time-like future directed curve $\lambda : [a,b] \rightarrow I^+(S)$ such that
 $$\lambda(a)=p$$ and $$\lambda(b)=q.$$

 Moreover for each $U \subset M$ open in $M$, $I^+(p)|_U$ consists of all point reach by time like future directed geodesics starting in $p$ and contained in $U$, so that $I^+(p)|_U$ denotes the chronological future of the space-time $U \subset M.$
\end{dfn}

\begin{dfn} A set $S \subset M$ is said to be achronal  if does not exist $p,q \in S$ such that $q \in I^+(p),$ that is if $$I^+(S) \cap S= \emptyset.$$
\end{dfn}

\begin{thm} Let $M$ be a space-time manifold. Let $S \subset M$. Under such assumptions $\partial I^+(S)$ is achronal.
\end{thm}
\begin{proof}
Let $q \in \partial I^+(S).$ Assume $p \in I^+(q).$ Thus $q \in I^-(p)$ and since $I^-(p)$ is open in $M$ there exists $U$ open in $M$,
 such that $U \subset I^-(p)$ and also such that $q \in U$.

 Note that since $q \in \partial I^+(S)$ we have that $$U \cap I^+(S) \neq \emptyset.$$  Let $q_1 \in U \cap I^+(S)$.

 From this, there exists  $p_1 \in S$ and a continuous and piece-wise $C^1$ class time-like future directed curve $\lambda:[a,b] \rightarrow M$ such that
 $\lambda(a)=p_1$ and $\lambda(b)= q_1 \in U \subset I^-(p)$.

 From such a result we may obtain a continuous and piece-wise $C^1$ class time-like future directed curve $\lambda_1:[b,c] \rightarrow M$ such that
 $\lambda_1(b)=q_1$ and $\lambda_1(c)= p$ so that $\lambda_2:[a,c] \rightarrow M$ such that
  \begin{equation}\label{br100} \lambda_2(s)=\left\{
\begin{array}{ll}
\lambda(s),& \text{ if } s \in [a,b] \\
 \lambda_1(s),& \text{ if }  s \in [b,c]
 \end{array} \right. \end{equation}
is a continuous and piece-wise $C^1$ time-like future directed curve such that $\lambda_2(a)=p_1 \in S$ and $\lambda_2(c)=p.$

 Therefore, we may infer that $p \in I^+(S), \; \forall p \in I^+(q)$, so that $$I^+(q) \subset I^+(S).$$

 Suppose, to obtain contradiction, that $\partial I^+(S)$ is not achronal.

 Thus, there exist $q,r \in \partial I^+(S)$ such that $$r \in I^+(q) \subset I^+(S).$$

 From this, we may infer that $$r \in \partial I^+(S) \cap I^+(S),$$ which contradicts $I^+(S)$ to be open.

 Therefore, $\partial I^+(S)$ is achronal.
\end{proof}
\begin{dfn} Let $M$ be a space-time manifold and let $\lambda \subset M$ be a causal future directed curve. We say that a point $p \in M$ is a final point of $\lambda$ if for each open set $U$ such that $p \in U$, there exists $s_0 \in \mathbb{R}$ such that if $s>s_0$, then  $$\lambda(s) \in U.$$

Moreover, we say that a curve is inextensible if does not have any final point.

Past inextensibility is defined similarly.
\end{dfn}
\begin{thm} Let $M$ be a closed space-time manifold. Let $$\lambda_n:(-\infty,b] \rightarrow M$$ be a sequence of differentiable past inextensible curves such that
for each $m \in \mathbb{N}$ there exists $K_m,\;\hat{K}_m \in \mathbb{R}^+$ such that
$$\|\lambda_n(s)\| \leq K_m,\; \forall s \in [-m,b],$$
 and
$$\|\lambda'_n(s)\| \leq \hat{K}_m,\; \forall s \in [-m,b].$$

Assume there exists $p \in M$ that for each open $U$ such that $p \in U$, there exists $n_0 \in \mathbb{N}$ such that if $n>n_0$ then there exists $s_n \in [-\infty,b),$  such that
 $$\lambda_n(s) \subset U,\; \forall s \in (s_n,b].$$

Under such hypotheses, there exist a subsequence $\{\lambda_{n_k}\}$ of $\{\lambda_n\}$ and a continuous curve $\lambda:(-\infty,b] \rightarrow M$
such that $$\lambda_{n_k} \rightarrow \lambda, \text{ uniformly } \text{ in } [-m,b],\; \forall m \in \mathbb{N}$$

and $$\lambda(b)=p.$$
\end{thm}
\begin{proof} Let $$\{\alpha_n\}=\{ q \in \mathbb{Q}\;:\; q \leq b\}.$$

Observe that,  from the hypotheses $\{\lambda_n(\alpha_1)\} \subset M$ is a bounded sequence, so that there exists a subsequence $$\{\lambda_{n_k}(\alpha_1)\}$$
and a vector which we shall denote by $\lambda(\alpha_1)$ such that
$$\lambda_{n_k}(\alpha_1) \rightarrow \lambda(\alpha_1), \text{ as } k \rightarrow \infty.$$

We shall also denote $$\lambda_{n_k}(\alpha_1)=L^1_k(\alpha_1).$$

Similarly $\{L_1^k(\alpha_2)\}$ is bounded so that there exists a subsequence $\{L^1_{n_k}\}$ of $\{L^1_k\}$ and a vector in $M$, which we will denote by
$\lambda(\alpha_2)$ such that $$L_{n_k}^1(\alpha_2)\rightarrow \lambda(\alpha_2), \text{ as } k \rightarrow \infty.$$

Denoting $L_{n_k}^1=L_k^2$ we have obtained
$$L_{k}^2(\alpha_1) \rightarrow \lambda(\alpha_1),$$
and
$$L_{k}^2(\alpha_2) \rightarrow \lambda(\alpha_2), \text{ as } k \rightarrow \infty.$$

Proceeding in this fashion, we may inductively obtain a subsequences $\{L_{k}^j\}_{k \in \mathbb{N}}$ of $\lambda_n$ such that
$$L_k^j(\alpha_l) \rightarrow \lambda(\alpha_l) \text{ as } k \rightarrow \infty,\; \forall l \in \{1,\ldots, j\}.$$

Let $\varepsilon>0$, $l \in \mathbb{N}$  and $j \geq l.$ Hence there exists $K_j \in \mathbb{N}$ such that if $k \geq K_j$ then
$$\|L_k^j(\alpha_l)-\lambda(\alpha_l)\| < \varepsilon.$$

In particular $$\|L_{K_j}^j(\alpha_l)-\lambda(\alpha_l)\|< \varepsilon, \forall j> l.$$

Hence, denoting $$\Lambda_j=L_{K_j}^j,$$ we have obtained that $\{\Lambda_j\}$ is a subsequence of $\{\lambda_n\}$ such that
$$\Lambda_k(\alpha_j) \rightarrow \lambda(\alpha_j),\; \forall j \in \mathbb{N}.$$

Fix $m \in \mathbb{N}$ such that $-m<b$ and let $s \in [-m,b].$ We are going to prove that $$\{\Lambda_k(s)\}$$ is a Cauchy sequence.

Let $\{\alpha_{n_l}\}$ be a subsequence of $\{\alpha_n\}$ such that $$\alpha_{n_l} \rightarrow s, \text{ as } l \rightarrow \infty.$$

Hence, there exists $l_0 \in \mathbb{N}$ such that if $l>l_0$, then
$$|\alpha_{n_l}-s|< \frac{\varepsilon}{3\hat{K}_m}.$$

Choose $l>l_0$. Since $\{\Lambda_k(\alpha_{n_l})\}$ is a Cauchy sequence, there exists $k_0 \in \mathbb{N}$ such that if $k,p> k_0$, then
$$\|\Lambda_k(\alpha_{n_l})-\Lambda_{p}(\alpha_{n_l})\| <\frac{\varepsilon}{3}.$$

Thus, if $k,p > k_0$, we obtain
\begin{eqnarray}
&&\|\Lambda_k(s)-\Lambda_p(s)\|\nonumber \\ &=& \|\Lambda_k(s)-\Lambda_k(\alpha_{n_l})+\Lambda_k(\alpha_{n_l})-\Lambda_p(\alpha_{n_l})
+\Lambda_p(\alpha_{n_l})-\Lambda_p(s)\| \nonumber \\ &\leq&
\|\Lambda_k(s)-\Lambda_k(\alpha_{n_l})\|+\|\Lambda_k(\alpha_{n_l})-\Lambda_p(\alpha_{n_l})\|
+\|\Lambda_p(\alpha_{n_l})-\Lambda_p(s)\| \nonumber \\ &\leq& \hat{K}_m |s-\alpha_{n_l}|+ \frac{\varepsilon}{3}+ \hat{K}_m |s-\alpha_{n_l}|
\nonumber \\ &<& \frac{\varepsilon}{3}+\frac{\varepsilon}{3}+\frac{\varepsilon}{3} \nonumber \\ &=& \varepsilon.
\end{eqnarray}

From this we may infer that $\{\Lambda_k(s)\}$ is a Cauchy sequence so that we may define
$$\lambda(s)=\lim_{k \rightarrow \infty} \Lambda_k(s),\; \forall s \in [-m,b].$$

We claim that this last convergence, up to a subsequence,  is uniform on $[-m,b].$

Indeed, let $$c_k=\sup_{s \in [-m,b]}\{\|\Lambda_k(s)-\lambda(s)\|\}.$$

Let $s_k \in [-m,b]$ be such that
$$ c_k-1/k < \|\Lambda_k(s_k)-\lambda(s_k)\| \leq c_k.$$

Since $[-m,b]$ is compact, there exist a subsequence $\{s_{k_l}\}$ of $\{s_k\}$ and $s \in [-m,b]$ such that
$$s_{k_l} \rightarrow s, \text{ as } l \rightarrow \infty.$$

At this point we shall prove that $$\|\lambda(s_{k_l})-\lambda(s)\| \rightarrow 0.$$

Indeed, there exists $l_0 \in \mathbb{N}$ such that if $l>l_0$, then $$|s_{k_l}-s|< \frac{\varepsilon}{\hat{K}}.$$

 $$\|\Lambda_p(s_{k_l})-\Lambda_p(s)\| \leq \hat{K}|s_{k_l}-s|<  \varepsilon, \forall l>l_0,\;\forall p \in \mathbb{N}.$$

 From this, we get
 $$\|\lambda(s_{k_l})-\lambda(s)\|=\lim_{p \rightarrow \infty} \|\Lambda_p(s_{k_l})-\Lambda_p(s)\| \leq \varepsilon,\; \forall l >l_0.$$

 Observe that from such a result we may, in a similar fashion, infer that $\lambda$ is continuous.

 From these last results, observing that  there exists $l_1 \in \mathbb{N}$ such that if $l>l_1$, then
 $$\|\Lambda_{k_l}(s)-\lambda(s)\|< \varepsilon,$$ we have that
 \begin{eqnarray}
 && \|\Lambda_{k_l}(s_{k_l})-\lambda(s_{k_l})\| \nonumber \\ &=& \|\Lambda_{k_l}(s_{k_l})-\Lambda_{k_l}(s)+\Lambda_{k_l}(s)-\lambda(s)+\lambda(s)-\lambda(s_{k_l})\|
 \nonumber \\ &\leq&\|\Lambda_{k_l}(s_{k_l})-\Lambda_{k_l}(s)\|+\|\Lambda_{k_l}(s)-\lambda(s)\|+\|\lambda(s)-\lambda(s_{k_l})\|
 \nonumber \\ &\leq& \varepsilon+\varepsilon+\varepsilon \nonumber \\ &=& 3 \varepsilon,\; \forall l > \max\{l_0,l_1\}.\end{eqnarray}

From this we may infer that $c_{k_l} \rightarrow 0$ as $l \rightarrow \infty,$ so that the convergence in question of the subsequence $\{\Lambda_{k_l}\}$ of
 $\{\lambda_n\}$ is uniform.
We claim now that $c_k \rightarrow 0$ as $k \rightarrow \infty.$

Suppose, to be contradiction, that the claim is false. Thus, $\{c_k\}$ does not converge to $0$.

Hence, there exists $\varepsilon_0>0$ such that for each $k \in \mathbb{N}$ there exists $k_l>k$ such that \begin{equation}\label{rt700}c_{k_l} \geq \varepsilon_0.\end{equation}

However, exactly as we have done with $\{c_k\}$ in the lines above, we may obtain a subsequence of $\{c_{k_l}\}$ which converges to $0$.
This contradicts (\ref{rt700}).

Therefore $$c_k \rightarrow 0, \text{ as } k \rightarrow \infty.$$

From this we may infer that $$\Lambda_k \rightarrow \lambda, \text{ uniformly in } [-m,b],\; \forall m \in \mathbb{N} \text{ such that } -m<b.$$

The proof is complete.
\end{proof}
\begin{thm} Let $M$ be a space time manifold. Assume that $\lambda:(-\infty,b] \rightarrow M$ is a causal future directed past inextensible curve.

Under such hypotheses, $$\lambda(s) \in \overline{I^+(\lambda)},\; \forall s \in (-\infty,b].$$
\end{thm}
\begin{proof} Let $s \in (-\infty,b]$ and choose $s_1<s.$

Thus, $\lambda|_{[s_1,s]}$ is a causal future directed curve such that denoting
$p=\lambda(s_1)$ and $q=\lambda(s)$, we have that
$$q \in \overline{I^+(p)} \subset \overline{I^+(\lambda)}, \forall s \in (-\infty,b].$$

The proof is complete.
\end{proof}

\begin{thm}\label{5.103} Let $M$ be a normal space time manifold.    Assume $\lambda :(-\infty,c] \rightarrow M$ is a causal future directed past inextensible curve which passes through a point $p \in M.$

Under such hypotheses, for each $q \in I^+(p)$ there exists a continuous and  piece-wise $C^1$ class time-like future directed past inextensible curve $\gamma : (-\infty,b] \rightarrow M$,  such that $$\gamma \subset I^+(\lambda)$$ and $$\gamma(b)=q.$$
\end{thm}
\begin{proof}
 Let $\hat{\lambda}: [a,b] \rightarrow I^+(\lambda)$ be a time-like future directed curve such that $$\hat{\lambda}(a)=p$$ and
$$\hat{\lambda}(b)=q.$$

We claim that $\hat{\lambda} \subset I^+(\lambda).$

Indeed, let $s \in (a,b].$ Denoting $q_1 =\hat{\lambda}(s)$ we have
$$\hat{\lambda}(s)=q_1 \in I^+(p) \subset I^+(\lambda),\; \forall s \in (a,b].$$
So the concerning claim holds.

Let $\lambda_1:(-\infty,b] \rightarrow M$ be the curve defined by
\begin{equation}\label{br100} \lambda_1(s)=\left\{
\begin{array}{ll}
\lambda(s),& \text{ if } s \in (-\infty,a] \\
 \hat{\lambda}(s),& \text{ if } a \leq s \leq b
 \end{array} \right. \end{equation}

Since the graph of $\lambda$ is connected and $M$ is normal, there exists $r>0$ such that if $\tilde{p},\tilde{q} \in I^+(\lambda)$ and $$0<\|\tilde{p}-\tilde{q}\| <r,$$ then renaming $\tilde{p},\tilde{q}$ if necessary, there exists a time-like future directed curve $\tilde{\lambda}:[c,d] \rightarrow I^+(\lambda)$ such that
$$\tilde{\lambda}(c)=\tilde{p}$$ and
$$\tilde{\lambda}(d)=\tilde{q}.$$

Let $\{s_n\} \subset (-\infty,b]$ be a real sequence such that
$s_1=b$, $s_{n} > s_{n+1}, \; \forall n \in \mathbb{N},$  $$\lim_{n \rightarrow \infty} s_n = -\infty.$$

and also such that $$\|\lambda_1(s_{n+1})-\lambda_1(s_n)\|<\frac{r}{3},\; \forall n \in \mathbb{N}.$$

Define $p_1=q$. Since $$\lambda_1 \subset \overline{I^+(\lambda)}$$ and $I^+(\lambda)$ is open,  for each $n > 1$ we may select $p_n \in  I^+(\lambda)$ such that $$0<\|p_n-\lambda_1(s_n)\|< \frac{r}{3}.$$

Observe that in such a case,
\begin{eqnarray}\label{p908}
\|p_{n+1}-p_n\|&=& \|p_{n+1}-\lambda_1(s_{n+1})+\lambda_1(s_{n+1})-\lambda_1(s_n)+\lambda_1(s_n)-p_n\| \nonumber \\ &\leq&
\|p_{n+1}-\lambda_1(s_{n+1})\|+\|\lambda_1(s_{n+1})-\lambda_1(s_n)\|+\|\lambda_1(s_n)-p_n\| \nonumber \\ &<&
\frac{r}{3}+\frac{r}{3}+\frac{r}{3} \nonumber \\ &=& r, \; \forall n \in \mathbb{N}.\end{eqnarray}

Moreover, $\{p_n\}$ may be chosen such that
$$0< d(p_n,\lambda)< \frac{C}{1+\sqrt{n}},\; \forall n \in \mathbb{N},$$ for some appropriate constant $C>0.$

Thus from (\ref{p908}) and from the fact that $M$ is normal, concerning such  $r>0$, we may obtain a  smooth time-like future directed curve $$\tilde{\lambda}_n:[s_{n+1}, s_{n}] \rightarrow  I^+(\lambda)$$ such that $$\tilde{\lambda}_n(s_{n+1})=p_{n+1},$$ and $$\tilde{\lambda}_n(s_n)=p_n.$$

Therefore, we may define $\gamma:(-\infty,b] \rightarrow M$ such that $$\gamma=\{\tilde{\lambda}_n:[s_{n+1},s_n] \rightarrow  I^+(\lambda)\;:\; n \in \mathbb{N}\},$$
which is a continuous and  piece-wise $C^1$ class time-like future directed  past inextensible curve such that $$\gamma(b)=q,$$ and $$\gamma \subset I^+(\lambda).$$
The  proof is complete.
\end{proof}
\begin{dfn} Let $M$ be a space time manifold. We say that $M$ is strongly causal if for each $p \in M$ and each neighborhood
$U$ of $p$, there exists a neighborhood $V$ of $p$ such that $V \subset U$ and no causal curve intersects $V$ more than one time.
\end{dfn}

\begin{thm} Let $M$ be a space-time manifold strongly causal. Let $K \subset M$ be a compact set. Under such hypotheses, each causal curve $\lambda$
contained in $K$ must have past and future final points.
\end{thm}
\begin{proof} Let $\lambda:[-\infty,+\infty] \rightarrow M$ be a causal curve contained in $K$. Let $\{s_j\} \subset \mathbb{R}$ be such that
$s_j < s_{j+1}$ and $$\lim_{j \rightarrow \infty} s_j= +\infty.$$

Observe that $$\{\lambda(s_j)=p_j\} \subset K$$ and $K$ is compact. Hence, there exists a subsequence $$\{p_{j_k}\}$$ and $p \in K$ such that
$$p_{j_k} \rightarrow p,\; \text{ as } k \rightarrow \infty.$$

Suppose, to obtain contradiction, we may obtain an open set $U$ such that $p \in U$ and such that for each $s_0 \in \mathbb{R}$ there exists $s>s_0$ such that
$\lambda(s) \not \in U.$ Thus we have the same for all $V \subset U$ such that $p \in V.$ Fixing an arbitrary  $V \subset U$ with $p \in V$, we have that $\lambda$ enters and leaves $V$ more than one time, because each time $\lambda$ enters $V$ it does not remain completely in $V$. Since $V \subset U$ has been arbitrary, this contradicts the strong
causality of $M$.

Thus, $p$ is a future final point for $\lambda.$ Similarly we may prove that $\lambda$ has a past final point.

This completes the proof.
\end{proof}
\section{ Dependence domains and hyperbolicity}

\begin{dfn} Let $S$ be a closed and achronal set. We define the domain of future dependence of $S$, denoted by $D^+(S)$, by
\begin{eqnarray}
D^+(S)&=& \{ p \in M\;:\;  \text{ each piece-wise smooth causal  future directed past inextensible curve } \nonumber \\ &&
\text{ which passes through } p \text{ intercepts } S\}. \end{eqnarray}
Observe that
$$S \subset D^+(S) \subset J^+(S),$$ and since $S$ is achronal, we have that
$$D^+(S) \cap I^-(S) = \emptyset.$$

The domain of past dependence of $S$, denoted by $D^-(S)$ is defined similarly.

We also define $$D(S)=D^+(S) \cup D^-(S),$$

Finally, an achronal set $\Sigma$ for which $D(\Sigma)=M$ is said to be a Cauchy surface for $M.$

Observe that, in such a case, $$\partial \Sigma= \emptyset.$$

Finally, a space-time manifold which has a Cauchy surface is said to be globally hyperbolic.
\end{dfn}
\begin{thm}Let $M$ be a normal space-time manifold and let $S \subset M$ be  a set closed in $M$. Under such hypotheses,  Let $p \in \overline{D^+(S)}$
if, and only if, each time-like future directed past inextensible curve which passes through $p$ intercepts $S$.
\end{thm}
\begin{proof} Suppose there exists  a time-like future directed past inextensible curve which does not intercept $S$.

Hence there exists a set $U$ open in $M$ such that $p \in U$ with such a propriety.

Thus $U \cap D^+(S) = \emptyset,$ so that $p \not \in \overline{D^+(S)}.$

Reciprocally, suppose each time-like future direct past inextensible curve which passes through $p$ intercepts $S$.

Thus, either $p \in S \subset D^+(S) \subset \overline{D^+(S)},$ and in such a case the proof would be finished, or $p \in I^+(S) \setminus S.$

In this latter case, let $q \in I^-(p) \cap I^+(S)$.

Suppose, to obtain contradiction, that $q \not \in D^+(S).$

Thus there exists a causal future directed past inextensible curve $\lambda$ which passes through $q$ and does not intercept $S$.

Note that $$\lambda \subset \overline{I^+(\lambda)}\setminus S,$$ so that, in such a case, since $M$ is normal, similarly as in the proof of Theorem \ref{5.103}, we may obtain  a piece-wise smooth time-like future directed past inextensible curve $\gamma$  such that
$$\gamma \subset I^+(\lambda),$$ also such that $\gamma \cap S= \emptyset$ and $\gamma$ passes through $p$,  which contradicts the hypotheses in question.

Hence, if $q \in I^-(p) \cap I^+(p),$ then $q \in D^+(S),$
so that $$I^-(p) \cap I^+(S) \subset D^+(S).$$
Since each neighborhood of $p \in I^+(S)$  intercepts $$I^-(p) \cap I^+(S) \subset D^+(S),$$ we have that
$p \in \overline{D^+(S)}.$

The proof is complete.
\end{proof}
\begin{thm} Let $M$ be a space-time  manifold. Let $S \subset M$ and let $\lambda:[a,b] \rightarrow M$ be a $C^1$ class future directed curve such that
$\lambda(s) \in \partial I^+(S),\; \forall s \in [a,b].$

Under such hypotheses, $\lambda$ is a null geodesics, that is,
$$\frac{d \lambda(s)}{d s} \cdot  \frac{d \lambda(s)}{d s}=0,\; \forall s \in [a,b].$$
\end{thm}
\begin{proof} Suppose, to obtain contradiction, that there exists $s_0 \in (a,b)$ such that $$\frac{d \lambda(s_0)}{d s} \cdot  \frac{d \lambda(s_0)}{d s}<0.$$
By continuity, there exists $\delta>0$ such that $$\frac{d\lambda(s)}{d s} \cdot  \frac{d \lambda(s)}{d s}<0,\; \forall s \in
(s_0-\delta,s_0+\delta).$$
Define $$p_1=\lambda(s_0),$$
and
$$p_2=\lambda\left(s_0+\frac{\delta}{2}\right).$$

Thus, $p_2 \in I^+(p_1),$ and $p_1,p_2 \in \partial I^+(S),$ which contradicts $\partial I^+(S)$ to be achronal.

Hence, $$\frac{d\lambda(s)}{d s} \cdot  \frac{d \lambda(s)}{d s} = 0,\; \forall s \in [a,b].$$

The proof is complete.
\end{proof}

\begin{thm} Let $M$ be a space-time manifold. Let $S \subset M$ and suppose $\lambda:[a,b] \rightarrow \overline{I^+(S)}$ is  a future directed null geodesics.

Under such hypotheses, $$\lambda(s) \in  \partial I^+(S),\; \forall s \in [a,b].$$
\end{thm}
\begin{proof}
Since $\lambda$ is future directed null geodesics, we have that
$$\frac{d \lambda(s)}{d s} \cdot  \frac{d \lambda(s)}{d s} = 0,\; \forall s \in [a,b].$$

From this, since  $\lambda(s) \in \overline{I^+(S)}$, we get $$\lambda(s) \in \partial I^+(S),\; \forall s \in [a,b].$$

The proof is complete.
\end{proof}
\begin{thm} Let $M$ be a normal space-time manifold. Let $\Sigma \subset M$ be a Cauchy surface and let $\lambda: [-\infty,+\infty] \rightarrow M$ be a causal
inextensible curve.

Under such hypotheses, $\lambda$ intercepts  $\Sigma$, $I^+(\Sigma)$ and $I^-(\Sigma).$
\end{thm}
\begin{proof}
Suppose, to obtain contradiction, that $\lambda$ does not intercept $I^-(\Sigma).$ Similarly as in the proof of  Theorem \ref{5.103}, we may obtain
a time-like past inextensible curve such that $$\gamma \subset I^+(\lambda) \subset I^+(\Sigma \cup I^+(\Sigma)) =I^+(\Sigma).$$

Extending $\gamma$ to the future indefinitely (if possible),  such a curve cannot intercept $\Sigma$, because in such a case $\Sigma$ would not be achronal, which is
contradiction.

However, since each causal inextensible curve must intercept $\Sigma$, we have got a final contradiction (that is, such a $\gamma$ does not exists).

From this we may infer that $\lambda$ intercepts $I^-(\Sigma).$

Similarly, we may show that $\lambda$ intercepts $I^+(\Sigma).$

The proof is complete.

\end{proof}
\section{Existence of solution for the previous general functional}

In this section, under some conditions, we prove the existence of solution for the general functional presented in the previous sections.
Specifically, we will be concerned with the existence of a kind of generalized solution for the main world sheet manifold.

We start with the following remark.
\begin{remark}  Considering the position field given by
$$\mathbf{r}: D=[0,T] \times D_1 \rightarrow \mathbb{R}^{N+1}$$
and fixing a small $\varepsilon >0$, define
\begin{eqnarray}U&=& \{\tilde{\mathbf{u}}=(\mathbf{r}, \phi,\mathbf{n}) \in C^2(\overline{D}; \mathbb{R}^{N+1}) \times C^1(\overline{D}; \mathbb{C})
\times  C^1(\overline{D}; \mathbb{R}^4)
\nonumber \\ &&  \text{ such that } |\phi|^2\geq \varepsilon \text{ in } D,\; \; \forall k \in \{0,\ldots,m\},
 \nonumber \\ &&\;\mathbf{r}(0,\mathbf{u})=\hat{\mathbf{r}}_0, \text{ in } D_1, \mathbf{r}(T,\mathbf{u})=\hat{\mathbf{r}}_1, \text{ in } D_1, \nonumber \\ &&
\mathbf{r}(t,\mathbf{u})=\hat{\mathbf{r}}_2, \text{ on } \partial D_1 \times [0,T]\nonumber \\ &&\phi(0,\mathbf{u})=\hat{\phi}_0, \text{ in } D_1, \phi(c T,\mathbf{u})=\hat{\phi}_1, \text{ in } D_1, \nonumber \\ &&
\phi(ct,\mathbf{u})=\hat{\phi}_2, \text{ on } \partial D_1 \times [0,T]
\},
\end{eqnarray}
\begin{eqnarray}
\tilde{U}&=& \{(\mathbf{r}, \phi,\mathbf{n}) \in W^{2,2}(D; \mathbb{R}^{N+1}) \times W^{1,2}(D; \mathbb{C})\times
 W^{1,2}(D; \mathbb{R}^4)
\nonumber \\ && \text{ such that }\mathbf{r}(0,\mathbf{u})=\hat{\mathbf{r}}_0, \text{ in } D_1, \mathbf{r}(T,\mathbf{u})=\hat{\mathbf{r}}_1, \text{ in } D_1, \nonumber \\ &&
\mathbf{r}(t,\mathbf{u})=\hat{\mathbf{r}}_2, \text{ on } \partial D_1 \times [0,T]
\nonumber \\ &&\phi(0,\mathbf{u})=\hat{\phi}_0, \text{ in } \Omega, \phi(c T,\mathbf{u})=\hat{\phi}_1, \text{ in } D_1, \nonumber \\ &&
\phi(ct,\mathbf{u})=\hat{\phi}_2, \text{ on } \partial D_1 \times [0,T]
\},
\end{eqnarray}

$$U_1=\left\{ \tilde{\mathbf{u}} \in U\;:\; \int_{D_1} |\phi(ct,\mathbf{u})|^2\sqrt{-g}\; d\mathbf{u}=1, \text{ on } [0,T]\right\},$$

$$U_2=\left\{ \tilde{\mathbf{u}} \in U\;:\; \frac{\partial \mathbf{r}(\mathbf{u})}{\partial u_j} \cdot \mathbf{n}(\mathbf{u})=0, \text{ in } D\right\},$$
and
$$U_3=\{\tilde{\mathbf{u}} \in U_1 \times U_2\;:\; \mathbf{n}(\mathbf{u})\cdot \mathbf{n}(\mathbf{u})=1,
\text{ in } D\}.$$

Finally, we define also,
$$A=U \cap U_1 \cap U_2$$
and $$A_1=\tilde{U} \cap U_1\cap U_2.$$
\end{remark}
With such definitions in mind we state and prove the following existence theorem.

\begin{thm} For $5 \leq m \leq 8$ and $m<N$, let $J_K:U \rightarrow \mathbb{R}$ be defined by
\begin{eqnarray}J_K(\tilde{\mathbf{u}})&=&J(\tilde{\mathbf{u}})+\frac{K}{2}\int_0^T\left(\int_{D_1} |\phi|^2 \sqrt{-g} \; d \mathbf{u}-1\right)^2\;dt
\nonumber \\ && +\frac{K}{2} \sum_{j=0}^m \int_D \left(\frac{\partial \mathbf{r}(\mathbf{u})}{\partial u_j} \cdot \mathbf{n}(\mathbf{u})\right)^2\sqrt{-g} \sqrt{U}\;d \mathbf{u}dt  \nonumber \\ &&
+\frac{K}{2}\int_D (\mathbf{n}(\mathbf{u})\cdot \mathbf{n}(\mathbf{u})-1)^2\sqrt{-g} \;d \mathbf{u}dt,
\end{eqnarray}
where \begin{eqnarray}
J(\tilde{\mathbf{u}})&=&
+\frac{1}{2} \int_D |\phi|^2g^{jk} b_{jl}b_k^l\sqrt{-g} \;d\mathbf{u}\;dt \nonumber \\ &&
+ \frac{1}{2}\int_D g^{jk} \frac{\partial \phi}{\partial u_j} \frac{\partial \phi^*}{\partial u_k} \sqrt{-g}\;d\mathbf{u}\;dt
\nonumber \\ && +\frac{1}{4} \int_D \left(\frac{\partial \phi}{\partial u_l}\phi^*+\frac{\partial \phi^*}{\partial u_l}\phi \right)\Gamma_{jk}^l g^{jk} \sqrt{-g}\;d\mathbf{u}\;dt,\end{eqnarray}
and where $K \in \mathbb{N}$ is a large constant.
Let $\{\tilde{\mathbf{v}}_n^K\}$ be a minimizing sequence for $J_K$, such that
$$\alpha \leq J_K(\tilde{\mathbf{v}}_n^K) < \alpha+\frac{1}{n},$$
where $$\alpha=\inf_{\tilde{\mathbf{v}} \in U} J_K(\tilde{\mathbf{u}}).$$

Suppose such a sequence is such that
\begin{enumerate}
\item There exists $c_0>0$ such that
$$(g^{jk})_n^K y_j y_k \geq c_0 y_j y_j,\;\forall y=\{y_j\} \in \mathbb{R}^2,\; \forall n \in \mathbb{N}, \text{ in } D.$$
\item There exists $c_1>0$ such that 
\begin{eqnarray}&&|\phi_n^K|^2(g^{jk})_n^K \mathbf{z}_j \cdot (\mathbf{g}_l)_n^K (g^{ls})_n^K \mathbf{z}_s \cdot (\mathbf{g}_k)_n^K \geq c_1 \mathbf{z}_j \cdot \mathbf{z}_j, \nonumber \\ &&
\forall \{\mathbf{z}_j\} \in \mathbb{R}^{(N+1)(m+1)},\;  \text{ in } D,\forall n \in \mathbb{N},\end{eqnarray}
so that
$$|\phi_n^K|^2(g^{jk})_n^K (b_{jl})_n^K (b_k^l)_n^K \geq c_1 \frac{\partial \mathbf{n}_n^K}{\partial u_i}\cdot  \frac{\partial \mathbf{n}_n^K}{\partial u_i}, \text{ in } D, \;\forall n \in \mathbb{N}.$$

 \item There exists $\{(c_2)_{ij}\}$ such that $(c_2)_{ij}>0,\; \forall i,j \in \{0,\ldots,m\},$ so that
$$|\phi_n^K|^2(g^{jk})_n^K (b_{jl})_n^K (b_k^l)_n^K \geq (c_2)_{ij} \left|\frac{\partial^2 \mathbf{r}_n^K}{\partial u_i \partial u_j}\right|^2, \text{ in } D,\; \forall n \in \mathbb{N}.$$
\item $$\|(\mathbf{g}_k)_n^K\|_{C^{1,\nu}(\overline{D})} \leq \hat{K},\;\forall k \in \{0,\ldots,m\},\; \forall n \in \mathbb{N},$$ for some $\hat{K} \in \mathbb{R}^+$ and some $0<\nu< 1$.
\end{enumerate}
Moreover, assume there exists $K_0 \in \mathbb{N}$, such that if $K>K_0$, then there exists $n_0 \in \mathbb{N}$ such that if $n > n_0$,
then
$$\int_{D_1} |\phi_n^K|^2 (\sqrt{-g})_n^K\;d\mathbf{u} \geq \frac{1}{4}, \text{ on } [0,T]$$
and
$$\mathbf{n}_n^K \cdot \mathbf{n}_n^K \geq \frac{1}{4},\; \text{ in } D.$$

Under such hypotheses, there exists $\tilde{\mathbf{u}}_0^K \in \tilde{U}$ such that
$$J_K(\tilde{\mathbf{u}}_0^K)=\inf_{ \tilde{\mathbf{u}} \in U} J_K(\tilde{\mathbf{u}}).$$

Finally, there exists a subsequence $\{K_j\}$ of $\mathbb{N}$ and  $\tilde{\mathbf{u}}_0 \in A_1$ such that
$$J(\tilde{\mathbf{u}}_0)=\lim_{j \rightarrow \infty} J_{K_j}(\tilde{\mathbf{u}}_0^{K_j})=\inf_{ \tilde{\mathbf{u}} \in A} J(\tilde{\mathbf{u}}).$$
\end{thm}
\begin{proof}
From the hypotheses we may infer that there exists $K_1 \in \mathbb{R}^+$ such that
\begin{enumerate}
\item $$ \left\|\frac{\partial^2 \mathbf{r}_n^K}{\partial u_i \partial u_j}\right\|_2 \leq K_1,\; \forall n \in \mathbb{N},
\; \forall i,j \in \{0,\ldots,m\}.$$
\item $$\|\phi_n^K\|_2 \leq K_1, \; \forall n \in \mathbb{N}.$$
\item $$\left\|\frac{\partial \phi_n^K}{\partial u_j}\right\|_2 \leq K_1, \; \forall n \in \mathbb{N},\; \forall j \in \{0,\ldots,m\}.$$
\item $$\left\|\frac{\partial \mathbf{n}_n^K}{\partial u_j}\right\|_2 \leq K_1, \; \forall n \in \mathbb{N},\; \forall j \in \{0,\ldots,m\}.$$
\end{enumerate}

Observe that $J_K$ is lower semi-continuous so that, from the Ekeland variational principle there exists a sequence $\{\tilde{\mathbf{u}}_n^K\} \in U$  such that
$$\alpha \leq J_K(\tilde{\mathbf{u}}_{n}^K)< \alpha+\frac{1}{n},$$
$$\|\tilde{\mathbf{u}}_n^K-\tilde{\mathbf{v}}_n^K\|_U< \frac{1}{\sqrt{n}},$$
and
$$\|\delta J_K(\tilde{\mathbf{u}}_{n}^K)\|_U \leq \frac{1}{\sqrt{n}},\; \forall n  \in \mathbb{N},$$

From such a result and from the variation of $J_K$  in $\mathbf{n}$ we obtain that
$$\frac{\partial }{\partial u_j}\left((a_{ij}^l)_n^K \frac{\partial (\mathbf{n})_n^K}{\partial u_i}\right)=(f_l)^K_n, \text{ in } D,$$
for appropriate positive definite $\{(a_{ij}^l)^K_n\} $ of $C^1$ class and $(f_l)_n^K \in L^2,\; \forall l \in \{1,\ldots,n+1\}, \; i,j \in \{0,...,m\}.$

Thus, from the Theory of Elliptic Partial Differential Equations, we have that $\mathbf{n}_{n}^K \in W^{2,2}$ and, since $\{(a_{ij}^l)_n^K\}$ and $\{(f_l)^K\}$ are uniformly bounded in
$C^1$ and $L^2$, respectively, there exists $K_3 \in \mathbb{R}^+$ such that
$$\|\mathbf{n}_{n}^K\|_{2,2} \leq K_3,\; \forall n \in \mathbb{N}.$$

With such results, we may similarly obtain that
$$\|\phi_{n}^K\|_{2,2} \leq K_4,\; \forall l \in \mathbb{N},$$ for some $K_4 \in \mathbb{R}^+.$

From such results and the Rellich-Kondrachov Theorem, we may obtain a subsequence $\{n_l\}$ of $\mathbb{N}$ and $\tilde{\mathbf{u}}_0^K \in \tilde{U}$ such that
\begin{enumerate}
\item $$\phi_{n_l}^K \rightharpoonup \phi_0^K, \text{ as } l \rightarrow \infty, \text{ weakly in }  W^{2,2}.$$
\item $$\phi_{n_l}^K \rightarrow \phi_0^K, \text{ as } l \rightarrow \infty, \text{ strongly in }  W^{1,q},$$
\item $$\mathbf{r}_{n_l}^K \rightharpoonup \mathbf{r}_0^K, \text{ as } l \rightarrow \infty, \text{ weakly in }  W^{2,2}.$$
\item $$\mathbf{r}_{n_l}^K \rightarrow \mathbf{r}_0^K, \text{ as } l \rightarrow \infty, \text{ strongly in }  W^{1,q}.$$
\item $$\mathbf{n}_{n_l}^K \rightharpoonup \mathbf{n}_0^K, \text{ as } l \rightarrow \infty, \text{ weakly in }  W^{2,2}.$$
\item $$\mathbf{n}_{n_l}^K \rightarrow \mathbf{n}_0^K, \text{ as } l \rightarrow \infty, \text{ strongly in } W^{1,q},$$
\end{enumerate}
$\forall 1 \leq q \leq \frac{2m}{m-4}\equiv p^*.$
At this point, firstly we highlight that,
up to a not relabeled subsequence
$$\left|(\sqrt{-g})_{n_l}^K- (\sqrt{-g})_0^K\right|^4 \rightarrow 0, \text{ as } l \rightarrow \infty, \text{ a.e. in } D,$$
and
$$\left\|\left|(\sqrt{-g})_{n_l}^K- (\sqrt{-g})_0^K\right|^4 \right\|_\infty<\hat{K}_1,\; \forall l \in \mathbb{N},$$ for some appropriate
$\hat{K}_1 \in \mathbb{R}^+$, so that, from the Lebesgue Dominated Convergence Theorem,
 we have
 $$\|(\sqrt{g})_{n_l}^K-(\sqrt{-g}_0)^K\|_{4} \rightarrow 0, \text{ as } l \rightarrow \infty.$$
Thus,
\begin{eqnarray} &&\left|\int_D \frac{\partial \phi_{n_l}^K}{\partial u_j} \frac{\partial (\phi^*)_{n_l}^K }{\partial u_k} (g^{jk})_{n_l}^K (\sqrt{-g})_{n_l}^K\; d \mathbf{u}dt
\right.\nonumber \\ &&\left.
-\int_D \frac{\partial \phi_0^K}{\partial u_j} \frac{\partial (\phi^*)_0^K }{\partial u_k}(g^{jk})_0^K (\sqrt{-g})_0^K\; d \mathbf{u}dt\right|
\nonumber \\ &\leq& \int_D \left|\frac{\partial \phi_{n_l}^K}{\partial u_j} \frac{\partial (\phi^*)_{n_l}^K }{\partial u_k} (g^{jk})_{n_l}^K (\sqrt{-g})_{n_l}^K
\right.\nonumber \\ &&
\left.-\frac{\partial \phi_0^K}{\partial u_j} \frac{\partial (\phi^*)_0^K }{\partial u_k}(g^{jk})_0^K (\sqrt{-g})_0^K\;\right| d \mathbf{u}dt
\nonumber \\
&\leq& \int_D \left|\frac{\partial \phi_{n_l}^K}{\partial u_j} \frac{\partial (\phi^*)_{n_l}^K }{\partial u_k} (g^{jk})_{n_l}^K (\sqrt{-g})_{n_l}^K
\right.\nonumber \\ &&
\left.-\frac{\partial \phi_0^K}{\partial u_j} \frac{\partial (\phi^*)_{n_l}^K }{\partial u_k}(g^{jk})_{n_l}^K (\sqrt{-g})_{n_l}^K\;\right| d \mathbf{u}dt
\nonumber \\ &&+\int_D \left|\frac{\partial \phi_0^K}{\partial u_j} \frac{\partial \phi_{n_l}^K }{\partial u_k} (g^{jk})_{n_l}^K (\sqrt{-g})_{n_l}^K
\right.\nonumber \\ &&
\left.-\frac{\partial \phi_0^K}{\partial u_j} \frac{\partial (\phi^*)_0^K }{\partial u_k}(g^{jk})_{n_l}^K (\sqrt{-g})_{n_l}^K\;\right| d \mathbf{u}dt
\nonumber \\ &&+\int_D \left|\frac{\partial \phi_0^K}{\partial u_j} \frac{\partial (\phi^*)_0^K }{\partial u_k} (g^{jk})_n^K (\sqrt{-g})_{n_l}^K
\right.\nonumber \\ &&
\left.- \frac{\partial \phi_0^K}{\partial u_j} \frac{\partial (\phi^*)_0^K }{\partial u_k}(g^{jk})_0^K (\sqrt{-g})_{n_l}^K\;\right| d \mathbf{u}dt
\nonumber \\ &&+\int_D \left|\frac{\partial \phi_0^K}{\partial u_j} \frac{\partial (\phi^*)_0^K }{\partial u_k} (g^{jk})_0^K (\sqrt{-g})_{n_l}^K
\right.\nonumber \\ &&
\left.-\frac{\partial \phi_0^K}{\partial u_j} \frac{\partial (\phi^*)_0^K }{\partial u_k}(g^{jk})_0^K (\sqrt{-g})_0^K\;\right| d \mathbf{u}dt
\nonumber \\ &\leq&K_1
\|\phi_{n}^K-\phi_0\|_{1,4}\|(\phi^*)_{h_l}^K\|_{1,4}\|(g^{jk})_{n_l}^K\|_{4}\|(\sqrt{-g})_{n_l}^K\|_{4} \nonumber \\ &&
+K_1\|\phi_0^K\|_{1,4} \|(\phi^*)_{n_l}^K-(\phi^*)_0^K\|_{1,4}\|(g^{jk})_{n_l}^K\|_{4}\|(\sqrt{-g})_{n_l}^K\|_{4} \nonumber \\ &&
+K_1\|\phi_0^K\|_{1,4}^2\|(g^{jk})_{n_l}^K-(g^{jk})_0^K\|_{4}\|(\sqrt{-g})_{n_l}^K\|_{4} \nonumber \\ &&
+K_1\|\phi_0^K\|_{1,4}^2\|(g^{jk})_0^K\|_{4}\|(\sqrt{-g})_{n_l}^K-(\sqrt{-g}_0)^K\|_{4} \nonumber \\ &\rightarrow& 0, \text{ as } l \rightarrow \infty.
\end{eqnarray}
Similarly we may prove the continuity of the remaining functional parts, so that
$$J_K(\tilde{\mathbf{u}}_{n_l}^K) \rightarrow J(\tilde{\mathbf{u}}_0^K)=\min_{\tilde{\mathbf{u}} \in U} J_K(\tilde{\mathbf{u}}).$$

At this point we observe that, through the Euler-Lagrange equations, the hypotheses and the limit process, we have obtained
$$\int_{D_1} |\phi_0^K|^2 \sqrt{-g}_0^K\;d \mathbf{u}-1 = \mathcal{O}(1/K),\; \text{ on } [0,T],$$
$$\int_D \left(\frac{\partial \mathbf{r}_0^K}{\partial u_j} \cdot \mathbf{n}_0^K\right)^2\sqrt{-g_0^K}\;d\mathbf{u}dt =\mathcal{O}(1/K)$$
and
$$\int_D \left( \mathbf{n}_0^K\cdot \mathbf{n}_0^K-1\right)^2\sqrt{-g_0^K}\;d\mathbf{u}dt =\mathcal{O}(1/K).$$

Observe also that the previous estimates are valid also for the sequence $\{\tilde{\mathbf{u}}_0^K\}$ (the concerning constants do not depend on $K$) so that there exists
$\tilde{\mathbf{u}}_0 \in \tilde{U}$ such that, up to a not relabeled subsequence,
\begin{enumerate}
\item $$\phi_{0}^K \rightharpoonup \phi_0, \text{ as } K \rightarrow \infty, \text{ weakly in }  W^{2,2}.$$
\item $$\phi_{0}^K \rightarrow \phi_0, \text{ as } K \rightarrow \infty, \text{ strongly in }  W^{1,q}.$$
\item $$\mathbf{r}_{0}^K \rightharpoonup \mathbf{r}_0, \text{ as } K \rightarrow \infty, \text{ weakly in }  W^{2,2}.$$
\item $$\mathbf{r}_{0}^K \rightarrow \mathbf{r}_0, \text{ as } K \rightarrow \infty, \text{ strongly in }  W^{1,q}.$$
\item $$\mathbf{n}_{0}^K \rightharpoonup \mathbf{n}_0, \text{ as } K \rightarrow \infty, \text{ weakly in }  W^{2,2}.$$
\item $$\mathbf{n}_{0}^K \rightarrow \mathbf{n}_0, \text{ as } K \rightarrow \infty, \text{ strongly in }W^{1,q},$$
\end{enumerate}
where as previously indicated, $\forall 1 \leq q \leq p^*.$

Moreover from the previous estimates and concerning limits (obtained similarly as above indicated),
$$\int_{D_1} |\phi_0|^2 \sqrt{-g_0}\;d\mathbf{u}-1 = 0,\; \text{ on } [0,T],$$
$$\int_D \left(\frac{\partial \mathbf{r}_0}{\partial u_j} \cdot \mathbf{n}_0\right)^2\sqrt{-g_0} \;d\mathbf{u}dt =0$$
and
$$\int_D \left( \mathbf{n}_0\cdot \mathbf{n}_0-1\right)^2\sqrt{-g_0}\;d\mathbf{u}dt =0.$$
From this we get, $$\tilde{\mathbf{u}}_0 \in A_1,$$
so that
$$J(\tilde{\mathbf{u}}_0)=\lim_{K \rightarrow \infty} J_K(\tilde{\mathbf{u}}_0^K)=\inf_{\tilde{\mathbf{u}} \in A} J(\tilde{\mathbf{u}}).$$

The proof is complete.
\end{proof}
\section{Conclusion} In this article we have obtained a variational formulation for relativistic mechanics based on standard tools of differential geometry.
The novelty here is that the main manifold has its range in a space of dimension $n+1$. In such a formulation the concept of normal field plays a fundamental role.

In the second article part, we have presented some formalism concerning the causal structure in a general space-time manifold defined by a function
$$(\mathbf{r}\circ \hat{\mathbf{u}}): \Omega \times (-\infty,+\infty) \rightarrow \mathbb{R}^{n+1}.$$

It is worth highlighting the main reference for this second part is the book \cite{598}.

Finally, in the last section, we develop an existence result of a kind of generalized solution for the main manifold variational formulation.

\end{document}